\documentclass[7 pt]{article}   
\usepackage{t1enc}
\usepackage{lmodern}
\usepackage[T1]{fontenc}
\usepackage[english]{babel}
\usepackage{graphicx}
\usepackage{tocloft}
\usepackage{yhmath}
\usepackage{amsmath}
\usepackage{textcomp}
\usepackage{ytableau}
\usepackage{mathtools}
\usepackage{tensor}
\usepackage{relsize}

\usepackage{amsmath,amsthm, amsfonts, amssymb, amscd, enumerate,mathrsfs, graphicx, graphpap, curves, color}
\usepackage[all]{xy}
\usepackage{amsmath}

\theoremstyle{plain}
\newtheorem{theo}{Theorem}[section]

\newtheorem{prop}[theo]{Proposition}
\newtheorem*{prop*}{Proposition}
\newtheorem{conjecture}[theo]{Conjecture}

\theoremstyle{definition}

\newtheorem{remark}[theo]{Remark}
\newtheorem{definition}{Definition}

\newcommand{\s}{\sigma}

\newcommand{\bK}{\mathbb{K}}

\newcommand{\bN}{\mathbb{N}}

\newcommand{\bP}{\mathbb{P}}

\newcommand{\gl}{\mathfrak{l}}

\newcommand{\gs}{\mathfrak{s}}

\renewcommand{\square}{\kern1pt\vbox
{\hrule height 0.6pt\hbox{\vrule width 0.6pt\hskip 3pt \vbox{\vskip
6pt}\hskip 3pt\vrule width 0.6pt}\hrule height0.6pt}\kern1pt}

\newcommand{\beq}{\begin{equation}}
\newcommand{\eeq}{\end{equation}}

\newcommand{\ol}{\overline}

\newcommand{\be}{\begin{equation}}
\newcommand{\ee}{\end{equation}}

\def\<#1,#2>{\langle\,#1,\,#2\,\rangle}
\newcommand{\arr}{\begin{array}{rlll}}
\newcommand{\ea}{\end{array}}
\newcommand{\bea}{\begin{eqnarray}}
\newcommand{\eea}{\end{eqnarray}}
\newcommand{\bean}{\begin{eqnarray*}}
\newcommand{\eean}{\end{eqnarray*}}

\newcommand{\Math}{{\it Mathematica\raise5 pt\hbox{$\scriptscriptstyle \lrcornerledR$}7}}

\newtheorem*{lemma*}{Lemma}
\newcommand{\superwedge}{\mathlarger{\mathlarger{\wedge}}}

\def\sideremark#1{\ifvmode\leavevmode\fi\vadjust{
\vbox to0pt{\hbox to 0pt{\hskip\hsize\hskip1em
\vbox{\hsize3cm\tiny\raggedright\pretolerance10000
\noindent #1\hfill}\hss}\vbox to8pt{\vfil}\vss}}}

\newcounter{ssig}
\newcounter{ttig}
\title{A note on the maximal rank}
\author{Alessandra Bernardi and Reynaldo Staffolani}
\date{}

\begin{document}
\maketitle
\begin{abstract}
    We give an upper-bound for the $X$-rank of points with respect to a non-degenerate irreducible variety $X$ in the case that sub-generic $X$-rank points generate a hypersurface.
\end{abstract}

\section*{Introduction}

All along the paper we will always work with an algebraically closed  field $\bK$  of characteristic $0$ and a projective variety $X\subset \mathbb{P}^N$ which will be always assumed to be irreducible and non-degenerate. 
For a given point $P\in\bP^N$ there is the well defined notion of $X$-{\it rank} $r_X(P)$ of $P$ which is the least number of points of $X$ whose span contains $P$. Such a notion, before becoming part of the algebraic geometric language as $X$-rank (referring to the underlined variety $X$ cf. \cite{BB}), was previously used in the context of tensors (i.e. when $X$ parametrizes particular type of tensors) and better known in the applied world as {\it structured rank} putting the accent on the particular structure of the tensors (cf. e.g. \cite{ACCF}).
From the applied point of view  the knowledge of the maximal possible $X$-rank that an element in $\langle X\rangle$ may have (e.g. \cite{AJRS, FL1, FL2, EWZ}) turns out to be extremely important. This raises a very interesting pure mathematical problem: being able to give a sharp upper bound on the maximal $X$-rank $r_{max}$.
One natural bound over an algebraically closed field of characteristic $0$ is given by the codimension,  i.e. $r_{max} \leq \operatorname{codim} X + 1$  (c.f. e.g \cite[Prop. 5.1, p. 348]{LT}). The next important result valid for any irreducible non-degenerate variety $X$ is presented in \cite{BT}. Let $g$ be the so called {\it generic $X$-rank} i.e. the first integer such that the Zariski closure of the set of points of rank smaller or equal than $g$ fills $\langle X \rangle$ (this is again a very much studied value, nowadays there are also numerical algorithms to compute it in certain cases cf. eg. \cite{HOOS,BDHM}). Then  \cite[Theorem 1, p. 1022]{BT} shows that $r_{max} \leq 2g$. In the  case in which $X$ is not a hypersurface but the points of subgeneric rank generate a hypersurface, \cite[Theorem 6, p. 1024]{BT} proved that $r_{max} \leq 2g - 1$. 

In \cite[Theorems 3.7 and 3.9, p. 118]{BHMT} this last bound for the hypersurfaces cases is improved to $r_{max}\leq 2g-2$ in the special setting of $X$ being either a curve or a homogeneous variety.

There are some other bounds worth noting in the case of maximum {\it symmetric} rank, i.e. the $X$-rank when $X$ is a Veronese variety. The first two are due to  Jelisiejew  \cite[Corollary 6, p. 331]{Joachim} and to Ballico--De Paris \cite[Remark 4.18, p. 913]{BDP}, which both obtain a bound on the {\it open symmetric rank}, a higher notion of rank which is always greater or equal than the symmetric one (we will recall them in \eqref{joachim} and in \eqref{ballico:deparis} respectively). Eventually there is also a bound given by  \cite[Prop. 3.3, p. 28]{DeP2} (see \eqref{DeParis:2}) for Veronese surfaces. An asymptotic bound is also presented.

\bigskip 

In this note, we focus on the case of $X$ being a variety such that the Zariski closure of the points of subgeneric rank $\sigma_{g-1}(X)$ is a hypersurface. In Theorem \ref{main} we will show that the bound for $r_{max}$ can be reduced to $$r_{max}\leq r_{max,g-1} + 1,$$ where $r_{max,g-1}$ is the maximum $X$-rank attained on the variety $\s_{g-1}(X)$.

We end the paper by comparing our bound with the existing ones highlighting the cases where our bound give more accurate estimates than the known ones.

\section{Notation and main result} 

\begin{definition}  Let $X\subset\bP^N$ be an irreducible non-degenerate projective variety. The Zariski closure of the set $\sigma_{s}^0(X)$ of points of $\bP^N$ of $X$-rank at most $s$ is an irreducible projective variety called $s$-{\it th secant variety of $X$} and denoted $\s_s(X)$.
\end{definition}

Secant varieties are nested and  there exists an integer $g$ such that $\s_g(X)$ fills the ambient space:

$$ X \subset \s_2(X) \subset \dots \subset \s_i(X) \subset \dots \subset \sigma_g(X)= \bP^N. $$

\begin{definition} Let $X\subset\bP^N$ be an irreducible non-degenerate projective variety.
The least integer $g$ such that $\sigma_g(X)=\bP^N$ is the {\it generic} $X$-rank.
\end{definition}

The generic $X$-rank may not coincide with the maximum $X$-rank appearing in $\bP^N$. There are cases in which the generic is the same as the maximum $X$-rank and cases in which there exist points with $X$-rank greater than the generic one, for example points on tangent lines of a rational normal curve of degree $d>2$ (cf. \cite{Syl}). 
Hence one may seek for a bound for the maximum rank and one may try to see if it is attained or not.

\begin{prop}\label{reducible:intersection}Let $X\subset\bP^N$ be an irreducible non-degenerate projective variety of dimension $n$ and let $W$ be a hypersurface strictly containing $X$. Let $P \in \bP^N \setminus W$ be such that $r_X(P) \neq 2$, and let
$$Y_P : = \operatorname{cone}(P,X) = \bigcup_{x \in X} \langle x,P \rangle .$$
Then $Y_P\cap W$ is reducible of dimension $n$. 
Moreover there exists a line in $Y_P$ through $P$ that meets $X$ in one point only and  $W\setminus X$ in at least one other point.
\end{prop}

\begin{proof}
The cone $Y_P$ is irreducible of dimension $n+1$, then the components of the intersection $Y_P \cap W$ have dimension at least $n+1+N-1-N = n=\dim(X)$ (cf. \cite[Thm.1.24, p. 75]{Shafarevich}).
Actually, since $Y_P$ is not contained in $W$, those components have dimension exactly $n$. Clearly $X\subset Y_P\cap W$.
Assume for the sake of contradiction that $ Y_P\cap W=X$.  If this is the case, then every line contained in $Y_P$ and passing through $P$ meets $W$ 
only on $X$ and moreover such intersection is made by one point only, otherwise the point $P$ would have $X$-rank $2$ which is against our hypothesis. Now we show that the fact that every line $l\subset Y_P$ through $P$ meets $W$ only in one point of $X$, i.e.: 
\begin{equation}\label{one:point}
  l\cap W=Q\in X\end{equation}
leads to a contradiction.

\noindent Since $W$ is a hypersurface, it is cut out by a single homogeneous equation of degree $d>1$, $W=V (f)$. As just shown every line $\ol{PQ} \subset Y_P$, with $Q \in X$, meets $W$ only in $Q$. We can parametrize $\ol{PQ}$ as

$$ \ol{PQ}=\{sP+tQ\ :\ [s:t] \in \bP^1 \} \cong \bP^1 $$

\noindent where in this notation the point $Q$ is represented by the point $[0:1] \in \bP^1$. Substituting the coordinates $sP+tQ$ of $\ol{PQ}$ in the equation $f$ of the hypersurface we get a homogeneous equation of degree $d$ in two variables which must vanish only at the point $Q$, i.e. we have an equation of the form 

\beq \label{EqBella} f(sP+tQ) = k_Qs^d \eeq

\noindent for some constant $k_Q \in \bK$. We show that if this happens for every line $\ol{PQ}$, with $Q \in X$, then we will get a contradiction. Indeed, suppose that $f$ can be written as the polynomial

$$ f(x_0,\dots,x_N)= \sum_{(b_0,\dots,b_N) \in \bN^{N+1},\ b_0 + \dots+b_N=d} a_{(b_0,\dots,b_N)} x_0^{b_0} \dots x_N^{b_N}.$$
The coefficient of the monomial $s^{d-1}t$ after the substitution \eqref{EqBella} has to zero.  Moreover it turns out to be the directional derivative of $f$ at $P$ in the direction of $Q$:

\begin{multline}
     \label{monomio2}
Df_P(Q) =  a_{(b_0,\dots,b_N)} [
b_0 \cdot p_0^{b_0-1} p_1^{b_1} \dots p_N^{b_N} \cdot q_0 + b_1 \cdot p_0^{b_0} p_1^{b_1-1} p_2^{b_2} \dots p_N^{b_N} \cdot q_1 + \dots
\\ 
\dots + b_N \cdot p_0^{b_0} \dots p_{N-1}^{b_{N-1}} p_{N}^{b_N-1} \cdot q_N ].  
\end{multline}

Since for a fixed $P$ the form $Df_P$ is linear in $Q$, then $X$ is contained in the hyperplane $Df_P = 0$, which is non-trivial by the Euler formula and the fact that the evaluation of $f$ at $P$ is different from zero. This is in contradiction with non-degeneracy hypothesis on $X$. Hence there exists a line $l$ inside $Y_P$ containing $P$ that intersects $W$ in at least another point in $W \setminus X$. This concludes the proof.
\end{proof}

\begin{theo} \label{main} Let $X \subset \bP^N$ be a smooth non-degenerate projective variety of dimension $n$ and let $g$ be the generic $X$-rank. If  $\s_{g-1}(X)$ is a hypersurface, $g >2$, then
$$r_{\text{max}} \leq r_{max,g-1}+1 $$ where $r_{max,g-1}$ is the maximum $X$-rank achieved on $\s_{g-1}(X)$.
\end{theo}

\begin{proof}
Let $P \in \bP^N \setminus \s_{g-1}(X)$ and $Y_P : = \operatorname{cone}(P,X) = \bigcup_{x \in X} \langle x,P \rangle $ as in Proposition \ref{reducible:intersection} with $W=\s_{g-1}(X)$, where we have shown that there exists a line of $Y_P$ through $P$ which intersects $\sigma_{g-1}(X)$ in at least two distinct points, say $Q_1, Q_2$ such that $Q_1\in X$ and $Q_2\in \sigma_{g-1}(X) \setminus X$. Therefore $P\in \langle Q_1,Q_2\rangle$. If $Q_2\in \sigma_{g-1}^0(X)$ then $r_X(P)\leq g$, while if the components of $Y_P\cap\sigma_{g-1}(X)$ different from $X$ are all contained in $\sigma_{g-1}(X)\setminus\sigma_{g-1}^0(X) $ we can only say that $r_X(P)\leq r_{max,g-1}+1$.
\end{proof}

\medskip

We would like to point out some interesting consequences of this result.

\begin{remark}
Suppose that $X\subset \mathbb{P}^N$ is a non-degenerate irreducible variety such that the last non filling secant variety $\sigma_{g-1}(X)$ is hypersurface, with $g \neq 2$, and let $P \in \mathbb{P}^N \setminus \sigma_{g-1}(X)$. The intersection $Y_P \cap \s_{g-1}(X)$ cannot be contained in $\s_{g-2}^0(X)$ otherwise 
the point $P \in \bP^N \setminus \s_{g-1}(X)$ must lie in $ \s_{g-1}(X)$ which is impossible.
\end{remark}

\begin{remark}\label{rem.pre.conj}
Suppose that $X\subset \mathbb{P}^N$ is a non-degenerate irreducible variety such that the last non filling secant variety $\sigma_{g-1}(X)$ is a hypersurface, with $g \neq 2$, and let $P \in \mathbb{P}^N \setminus \sigma_{g-1}(X)$. The rank of the point $P \in \bP^N \setminus \s_{g-1}(X)$ is $g$ if and only if the intersection $Y_P \cap \s_{g-1}(X)$ contains at least one point of $\s_{g-1}^0(X)$ but we were not able to distinguish whether there exist points $P$ for which $(Y_P\cap \sigma_{g-1}(X))\subset (\sigma_{g-1}(X)\setminus \sigma_{g-1}^0(X))
$. Of course if $P$ is generic it is obvious that $Y_P\cap \sigma_{g-1}^0(X)\neq \emptyset$, and indeed the generic rank is $g$. 

This leads us to the following conjecture.
\end{remark}

\begin{conjecture} \label{Conjecture}
Let $X \subset \bP^N$ be a smooth non-degenerate projective variety and let $g>2$ be the generic $X$-rank. If  $\s_{g-1}(X)$ is a hypersurface, then
$$r_{\text{max}} = \max\{ r_{max,g-1},g \}$$ where $r_{max,g-1}$ is the maximum $X$-rank achieved on $\s_{g-1}(X)$.
\end{conjecture}


\section{Comparison} \label{esempi} \medskip
In this section we compare our bound with the existing ones on some known examples. \smallskip

The most studied case is the one of symmetric tensors where  the rank is  with respect to a Veronese variety. Let $X_{n,d}$ be the Veronese embedding of $\mathbb{P}^n$ via $\mathcal{O}(d)$.
%
%
To the best of our knowledge, some of the best upper bounds for the Waring rank are due to J. Jelisiejew  \cite[Corollary 6, p. 331]{Joachim} and to E. Ballico-A. De Paris \cite[Remark 4.18, p. 913]{BDP}. In both these works, the bounds are given on a different notion of rank known as {\it open Waring rank}, we refer to  \cite[Definition 2, p. 330]{Joachim} for a definition. Since the open Waring rank is always greater or equal then the usual Waring rank, from \cite{Joachim} and \cite{BDP} one gets these bounds on the maximum symmetric rank 
\begin{equation}\label{joachim}
    r_{max}\leq \binom{n+d-1}{n}- \binom{n+d-5}{n-2}\end{equation}
for $n \geq 2,d \geq 3$ , and
\begin{equation}\label{ballico:deparis}
    r_{max}\leq \binom{n+d-1}{n}-\binom{n+d-5}{n-2}-\binom{n+d-6}{n-2}
\end{equation}
for $n \geq2$, $d\geq 4$, respectively. Another known bound is given by  \cite[Prop. 3.3, p. 28]{DeP2} for all homogeneous polynomials of degree $d$ in $3$ variables. In this case \cite{DeP2} shows that

\beq \label{DeParis:2} r_{max} \leq \left \lfloor \frac{d^2+6d+1}{4} \right \rfloor. \eeq
Eventually, consider any Veronese variety $X_{n,d}$. As pointed out in \cite{BT}, it is worth noting that the bound which they give $r_{max} \leq 2g$ in the general case is asymptotically better than \eqref{joachim} and \eqref{ballico:deparis}, even though these last ones are better for small cases. 
%
%

\begin{table}[h!]
    \centering
    \begin{tabular}{|c|c|c|c|c|c|c|c|}\hline
         &$2g - 1$ in \cite{BT}  &$2g-2$ \cite{BHMT}& \cite{DeP2} &\cite{BDP}&\cite{Joachim}&our bound& $\mathbf{r_{max}}$\\ \hline
    $X_{2,2}$      &5&4&4&&&3&$\mathbf{3}$ \\ \hline
    $X_{2,3}$ &7&6&7&&9&6&$\mathbf{5}$\\ \hline
    $X_{2,4}$&11&10&10&17&18&8&$\mathbf{7} $\\ \hline
    \end{tabular}
    \caption{Comparison of bounds on the maximum Waring rank when a secant variety of a Veronese variety is a hypersurface and where the maximum rank is known. See \cite[Theorem 40, p. 48, and Theorem 44, p. 50]{BGI} and \cite{DeP}.}
    \label{tab:my_label}
\end{table}

\smallskip

For the skew-symmetric tensors case one has to study the rank with respect to Grassmann varieties $ Gr(\bP^k,\bP^n)\subset \bP(\superwedge^{k+1} V)$, $2k \leq n-1$. We checked among $1 \leq n,r \leq 500$ and, if the conjecture on defectiveness of secant varieties of Grassmannians holds (cf. \cite{BDdG, CGGgrass}, see also \cite{AOP, McG, Bor, BV}), we found only three cases in which there exists an $r$ such that $\sigma_r(X)$ is a hypersurface: $Gr(\bP^1,\bP^3)$ for $r=1$ (but in such a case our theorem does not apply), $Gr(\mathbb{P}^2, \mathbb{P}^6)$ for $r=3$, $Gr(\bP^7,\bP^{16})$ 
for $r= 333$. 

The case of $Gr(\bP^1,\bP^3)$ is trivial since elements of $\bigwedge^2\mathbb{C}^4$ are skew-symmetric matrices.
The second example is a defective case, the well known $\sigma_3(Gr(\mathbb{P}^2, \mathbb{P}^6))$ (cf. \cite{Sch} and \cite[Section 5]{AOP}). The maximum skew-symmetric rank of a point belonging to $\sigma_3(Gr(\mathbb{P}^2, \mathbb{P}^6))$ is 3, so our main theorem shows that the maximum skew-symmetric rank of a point in $\mathbb{P}(\bigwedge^3\mathbb{C}^7)$ is the generic one, i.e. 4. Indeed in \cite{ABMM} it is shown that the maximum rank is actually 4. This is an example that shows the sharpness of our result (remark that \cite[Theorem 6, p. 3]{BT} in this case provides a bound of 7, while \cite[Theorem 3.9, p. 118]{BHMT} gives $6$).
%
%
%

\smallskip

For the case of Segre variety we highlight only the example of $X$ being the Segre of 3 copies of $\mathbb{P}^2$ where the maximum rank is known to be 5 (c.f. \cite[Thm.5.1, p. 412]{BrHu}, \cite[Thm.4, p.815]{MSS}) and our Theorem \ref{main} gives a bound of 6; this is not sharp but it is better than  \cite[Theorem 6, p. 3]{BT} which gives $r_{max}\leq 9$. The reason why we highlight this example is that in this case our Conjecture \ref{Conjecture} holds and by Remark  \ref{rem.pre.conj} the intersection $Y_P \cap \s_{4}(X)$ contains at least one point of $\s_{4}^0(X)$. 

\smallskip

For the case of Flag varieties we underline only the case of the adjoint varieties of the Lie Algebra $\gs \gl_{n+1}$. In this case the Flag variety contained in $\bP(\gs \gl_{n+1}) \simeq \bP^{(n+1)^2-2}$ is the variety $F(\bP^0,\bP^{n-1};\bP^n)$ whose points are flags $\bP^0 \subset \bP^{n-1}$ in $\bP^n$. It is a known fact that for any $n \geq 1$ this variety parametrizes the $(n+1) \times (n+1)$-traceless matrices of rank $1$. As showed by \cite[Theorem 1.1]{BD}, the $k$-th secant variety of this Flag variety is given by traceless matrices of rank at most $k$, for any $k \leq n+1$. See also \cite{PT} for a description of this in terms of lower semi-continous rank function. Moreover by \cite[Corollary 1.2]{BD}, it is easy to see that $\s_{n+1}(F(\bP^0,\bP^{n-1};\bP^n))$ fills the ambient space, and that the dimension of the $k$-th secant variety of $X$ is $2k(n+1) - k^2-2$. In particular for $k = n$ it is readily seen that $\s_{n}(F(\bP^0,\bP^{n-1};\bP^n))$ is a hypersurface in $\bP(\gs \gl_{n+1})$. By what we have said we get that $r_{max,n} = n$. By the Theorem \ref{main} we get that $r_{max} \leq n+1$. Applying the bound in \cite[Theorem 6, p. 3]{BT} one gets $r_{max} \leq 2(n+1)-1 = 2n + 1$, while by \cite[Theorem 3.9, p. 118]{BHMT} one gets $r_{max} \leq 2n$. Note that also in this case the Conjecture \ref{Conjecture} holds.

\section*{Acknowledgements} We thank E. Ballico, J. Buczy\'nski and the anonimous referees for useful comments. Both authors are partially supported by GNSAGA of INDAM.


\begin{thebibliography}{10}

\bibitem{AOP} H. Abo, G. Ottaviani and C. Peterson, {\it Non-defectivity of Grassmannians of planes}, J. Algebraic Geom., 21, (2012), 1--20.

\bibitem{ACCF} L. Albera, P. Chevalier, P. Comon and A. Ferreol, {\it On the virtual array concept for higher order array processing} IEEE Trans. Sig. Proc., 53, (2005), 1254--1271.

\bibitem{AJRS} E.S. Allman, P.D. Jarvis, J.A. Rhodes and J.G. Sumner, {\it Tensor rank, invariants, inequalities, and applications}, SIAM J. Matrix Anal. Appl., 34, (2013), 1014--1045.

\bibitem{Allums}   D. J. Allums, {\it Toward a classification of the ranks and border ranks of all $(3,3,3)$ trilinear forms}, junior thesis, Texas A\&M University, (2011).

\bibitem{ABMM}
E. Arrondo, A. Bernardi, P.M. Marques and B. Mourrain, {\it 
Skew-symmetric tensor decomposition}, Communications in Contemporary Mathematics, 23, (2021), 2050008.

\bibitem{BB} E. Ballico and A. Bernardi, {\it On the X-rank with respect to linear projections of projective varieties},
Mathematische Nachrichten, 284, (2011),  2133--2140.

\bibitem{BDP} E. Ballico and A. De Paris, {\it Generic power sum decompositions and bounds for the
Waring rank}, Discrete Comput. Geom., 57, (2017), 896--914.


\bibitem{BDdG} K. Baur, J. Draisma and W.A. de Graaf. {\it Secant dimensions of minimal orbits: Computations and conjectures}. Exp. Math., 16, (2007), 239--250.

\bibitem{BD} K. Baur and J. Draisma, {\it Higher secant varieties of the minimal adjoint orbit}. Journal of Algebra, 280, (2004), 743--761.
	
\bibitem{BDHM} A. Bernardi, N.S. Daleo, J.D. Hauenstein and B. Mourrain, {\it Tensor decomposition and homotopy continuation}, Differential Geometry and its Applications, 55, (2017),  78--105.

\bibitem{BGI} A. Bernardi, A. Gimigliano and M. Id\`a, {\it Computing symmetric rank for symmetric tensors}, Journal of symbolic computation, 46, (2011), 34-53.

\bibitem{BV} A. Bernardi and D. Vanzo, {\it A new class of non-identifiable skew-symmetric tensors}, 
Annali di Matematica Pura ed Applicata, 197, (2018), 1499--1510.

\bibitem{BT} G. Blekherman and Z. Teitler, {\it On maximum, typical and generic ranks},  Matematische Annalen, 362, (2015),
1021--1031.

\bibitem{BS} G. Blekherman and R. Sinn, {\it Real rank with respect to varieties},  Linear algebra and its applications, 505, (2016),
344--360.
	
\bibitem{BrHu} M. R. Bremner and J. Hu, {\it On Kruskal’s theorem that every 3 × 3 × 3 array has rank at most 5},  Linear algebra and its applications, 439, (2013), 401--421.	

\bibitem{BHMT} J. Buczy\'nski, K. Han, M. Mella and Z. Teitler, {\it On the locus of points of high rank}, European journal of mathematics, 4, (2018), 113-136.

\bibitem{Bor} A. Boralevi, {\it  A note on secants of Grassmannians}. Rend. Istit. Mat. Univ. Trieste, 45, (2013), 67--72.

\bibitem{CGGgrass}
M.V. Catalisano, A.V. Geramita and A Gimigliano, {\it Secant varieties of Grassmann varieties}. Proc. Am. Math. Soc. 133, (2005), 633--642.

\bibitem{DeP} A. De Paris, {\it A proof that the maximal rank for plane quartics is seven}, Le matematiche, 70, (2015), 3--18.

\bibitem{DeP2} A. De Paris, {\it The asymptotic leading term for maximum rank of ternary forms of even degree}, Linear algebra and its applications, 500, (2016), 15--29.

\bibitem{EWZ} M. Enríquez, I. Wintrowicz and K. Zyczkowski, {\it Maximally Entangled Multipartite States: A Brief Survey}, J. Phys.: Conf. Ser., 698, (2016), 012003.

\bibitem{FL1} B. Fortescue and H.-K. Lo, {\it Random bipartite entanglement from $W$ and $W$-like states}, Phys. Rev. Let., 98, (2007), no. 6, 260501.

\bibitem{FL2} B. Fortescue and H.-K. Lo, {\it Random-party entanglement distillation in multiparty states}, Phys. Rev. A 78, (2008), no. 1, 012348.

\bibitem{H} J. Harris, {\it Algebraic geometry, a first course}, Graduate Texts in Mathematics, vol. 133, Springer-Verlag, New York, (1995), Corrected reprint of the1992 original.

\bibitem{HOOS} J.D. Hauenstein,  L. Oeding, G. Ottaviani and A.J. Sommese, {\it Homotopy techniques for tensor decomposition and perfect identifiability}
Journal fur die Reine und Angewandte Mathematik, 753 (2019),   1--22.

\bibitem{Joachim} J. Jelisiejew, {\it An upper bound for the Waring rank of a form}, Archiv der Mathematik, 102, (2014), 329--336.

\bibitem{LT}J.M. Landsberg and Zach Teitler, {\it On the ranks and border ranks of symmetric tensors}, Found.
Comp. Math. 10 (2010), no. 3, 339--366.


\bibitem{McG} B. McGillivray, {\it A probabilistic algorithm for the secant defect of Grassmann varieties}. Linear Algebra Appl., 418, (2006), 708--718.

\bibitem{MSS} M. Miyazaki, T. Sakata and T. Sumi, {\it About the maximal rank of 3-tensors over the real and the complex number field}. Ann. Ist. Stat. Math., 62, (2010), 807--822.

\bibitem{Pa} F. Palatini, {\it Sulla rappresentazione delle forme ternarie mediante la somma di potenze di forme lineari}. Rend. Accad. Lincei, V, (1903), 378--384.

\bibitem{PT} A. V. Petukhov and V.V. Tsanov, {\it Homogeneous projective varieties with semicontinous rank function}, Manuscripta Mathematica, 147, (2015), 269--303
	
\bibitem{Sch}	J.J.A. Schouten, {\it Klassifizierung der alternierenden Groössendritten Grades in 7 Dimensionen}, Rend. Circ. Mat. Palermo, 55, (1931).
		
\bibitem{Shafarevich} I. Shafarevich, {\it Basic Algebraic Geometry 1}, Springer-Verlag Berlin Heidelberg, (2013).

\bibitem{Syl} J.J. Sylvester, {\it An essay on canonical forms, supplement to a sketch of a memoir on elimination,transformation and canonical forms}, originally published by George Bell, Fleet Street, London, (1851). Paper 34 in Mathematical Papers,  Vol. 1, Chelsea, New York, (1973), originally published by Cambridge University Press in 1904. (1851).


\end{thebibliography}
\end{document}